\documentclass[11pt]{amsart}
\usepackage{amssymb}
\usepackage{amsthm}

\usepackage{epsfig}
\usepackage{color}
\usepackage{verbatim}
\usepackage{tikz}

\usepackage{amsmath}
\usepackage{hyperref}
\usepackage{xypic}
\usepackage{amscd}
\usepackage{color}

\newfont{\sheaf}{eusm10 scaled\magstep1}

\newtheorem{thm}{Theorem}[section]
\newtheorem{cor}[thm]{Corollary}
\newtheorem{lemma}[thm]{Lemma}
\newtheorem{prop}[thm]{Proposition}
\newtheorem{defn}[thm]{Definition}
\newtheorem{claim}[thm]{Claim}

\newtheorem{proposition}[thm]{Proposition}

\theoremstyle{definition}
\newtheorem{remark}[thm]{Remark}

 \newtheorem{example}[thm]{Example}

\def\min{\operatorname{min}}

\def\max{\operatorname{max}}

\def\c1{\operatorname{c_1}}
\def\c2{\operatorname{c_2}}

\def\Sym{\operatorname{Sym}}

\def\e{\mathbf{e}}

\def\c{\mathbf{c}}

\def\ZZ{{\mathbb Z}}

\def\PP{{\mathbb P}}

\def\M{{\mathcal M}}

\def\H{{\mathcal H}}

\def\cong{\simeq}

\def\+{\oplus}                   
\def\*{\otimes}                  

\def\geq{\geqslant}
\def\leq{\leqslant}

\makeatletter
\@namedef{subjclassname@2020}{\textup{2020} Mathematics Subject Classification}
\makeatother

\begin{document}

\title{Components of simple and non--simple type of Hurwitz schemes}

\author[C.~Ciliberto]{Ciro Ciliberto}
\address{C.~Ciliberto, Dipartimento di Matematica,  Universit\`a di Roma ``Tor Vergata'', Via della Ricerca Scientifica, 00173 Roma, Italy}
\email{cilibert@axp.mat.uniroma2.it}

\author[A.~L.~Knutsen]{Andreas Leopold Knutsen}
\address{A.~L.~Knutsen, Department of Mathematics, University of Bergen,
Postboks 7800,
5020 Bergen, Norway}
\email{andreas.knutsen@math.uib.no}

\author[S.~Torelli]{Sara Torelli}
\address{S.~Torelli,  Dipartimento di Matematica, Politecnico di Milano, Piazza Leonardo da
Vinci 12, 20133 Milano,  Italy}
\email{sara.torelli7@gmail.com}

\begin{abstract}  Let $\H_{g \to b,d; \e}$, with $\e=(e_1,\ldots, e_n)$, be the \emph{Hurwitz space}, parametrizing all morphisms $\pi: C\to B$ of degree $d$, with $n$ points $x_1,\ldots, x_n\in C$ of ramification order  $e_1,\ldots, e_n$ respectively, and where $C$ and $B$ are smooth, irreducible, projective curves of genera $g$ and $b$ respectively. In this  paper  we study the question of when there exist components of $\H_{g \to b,d;\e}$ whose members $\pi: C \to B$ all factor through an intermediate curve, in which case we say that these components are  \emph{of non--simple type}.  We give necessary and sufficient conditions for the existence of  components of non--simple type.  Then we prove  that for $b\geq 2$ there are always
 components of simple type,  and  for $b\in \{0,1\}$ there are such components under suitable sufficient conditions. However there are easy examples for $b\in \{0,1\}$ in which there are never
 components of simple type.  
\end{abstract}

 \maketitle

\section{Introduction}

Consider, for any integers $d \geq 2$, $n,g,b \geq 0$, and any $n$-tuple of integers  $\e:=(e_1, \ldots, e_n)$  such that $2 \leq e_i \leq d$, the {\it Hurwitz space}
\[ \H_{g \to b,d;\e}. \]
 Letting 
$$r=r(g,b,d;\e):=2(g-1-d(b-1))-\displaystyle\sum_{i=1}^n(e_i-1),$$  
  the Hurwitz space  $\H_{g \to b,d;\e}$  is the coarse moduli space parametrizing   triples
\begin{equation}
  \label{eq:triple}
  \{(C,x_1,\ldots,x_n), B, \pi:C \to B\},
\end{equation}
where
\begin{itemize}
\item $C$ is a smooth  irreducible   projective  curve of genus $g$, and $x_1,\ldots, x_n \in C$ are distinct marked points, that is, $(C,x_1,\ldots,x_n) \in \M_{g,n}$;
\item $B$ is a smooth  irreducible    projective  curve of genus $b$,  that is, $B\in \M_b$;
  \item  $\pi:C \to B$ is a morphism of degree $d$ having ramification order $e_i$ at $x_i$, for $1\leq i\leq n$,  and simple ramification at  $r$   points completing the ramification profile. 
\end{itemize}
Sometimes we will use the shorter notation
\begin{equation}
  \label{eq:triples}
  [\pi:(C,x_1,\ldots,x_n) \longrightarrow B] \;\;\mbox{ or} \;\; [\pi:C \to B]
\end{equation}
instead of \eqref{eq:triple} for members of $\H_{g \to b,d;\e}$.

 Note that  $\H_{g \to b,d;\e}$ need not be irreducible. In fact, Theorems \ref{thm:composed}--\ref{thm:simple01} below provide infinitely many examples of reducible cases.

The Riemann-Hurwitz formula yields
\[ 2(g-1)=2d(b-1)+ \displaystyle\sum_{i=1}^n(e_i-1)+r, \]
whence in particular,
\begin{equation}
  \label{eq:RH}
   \displaystyle\sum_{i=1}^n(e_i-1)  \leq 2\left[g-1-d(b-1)\right] \;\;\;  (\mbox{and} \; 2 \leq e_i \leq d \;\; \mbox{for all $i$}). 
\end{equation}
We call \eqref{eq:RH} the {\it Riemann-Hurwitz condition}. Defining
\begin{equation}
  \label{eq:defsigma}
  \sigma(g,b,d;\e):= -\displaystyle\sum_{i=1}^n(e_i-1) -(2d-3)(b-1)-g+1,
\end{equation}
the Riemann-Hurwitz condition can be equivalently stated as
\begin{equation}
  \label{eq:RH_cond_s}
 \sigma(g,b,d;\e) \geq -3(g-b).
\end{equation}
We note that when $b=0$, then $\sigma(g,b,d;\e)=\widetilde{\rho}(g,1,d;\e)$  is    the {\it adjusted Brill-Noether number}  (see \cite {EH}).  

 The Hurwitz moduli space $\H_{g \to b,d;\e}$ is
nonempty if and only if  the Riemann-Hurwitz condition
 \eqref{eq:RH} holds, in which case it is 
equidimensional of dimension $3(g-1)+n+\sigma(g,b,d;\e)$, with a few exceptions, cf. Proposition \ref{prop:dimhur} below.

 It is worth remarking that if $[\pi:(C,x_1,\ldots,x_n) \longrightarrow B]$ is general in  a component of  $\H_{g \to b,d;\e}$, then $\pi(x_1), \ldots, \pi(x_n)$ are distinct points of $B$. This is an immediate consequence of Riemann's Existence Theorem  (cf., e.g., \cite[III, Cor. 4.9-10]{Mi}).

In this  paper  we study the question of when there exist components of $\H_{g \to b,d;\e}$ whose members $C \to B$ all factor through an intermediate curve. 
We say that a member $\pi:C \to B$ of $\H_{g \to b,d;\e}$ is {\it  non--simple}  if $\pi$  factors as
\begin{equation}
  \label{eq:compos}
  \xymatrix{
    C \ar[rr]^{\pi} \ar[dr]_{\phi} &  &  B  \\
   & C' \ar[ur]_{\psi} & 
}\end{equation}
where $C'$ is a smooth   irreducible  projective  curve  and $\phi$ and $\psi$ are not isomorphisms.   In this case one says that $\pi$ is \emph{composed} with $\phi$ and $\psi$ (or  only  composed with $\phi$).  Otherwise, we say that the member is {\it simple}.   Notice that if $\pi:C \to B$ is non--simple, the factorization \eqref {eq:compos}  need  not be unique.  

  We define a component of $\H_{g \to b,d;\e}$ to be {\it of  simple} (respectively, {\it non--simple})  {\it type} if its general member is  simple (resp., non--simple).  More precisely, we say that a component is {\it non--simple of order  $\delta>1$} if $\deg \phi=\delta$ for its general member (and consequently $\deg \psi=\frac{d}{\delta}$).  Note that, by what we observed above, the order of a component of $\H_{g \to b,d;\e}$ may not be unique. This will not be a problem for us. 
  We denote  by $\H_{g \to b,d;\e}^{\delta}$ the union of components of  non-simple type of order $\delta$ of $\H_{g \to b,d;\e}$.  

Our first main result yields necessary and sufficient conditions for the existence of components of non--simple type. To state it we need  the following: 

  \begin{defn} \label{def:adm}
  Given the data $(g,b,d;\e=(e_1,\ldots,e_n))$ satisfying the Riemann-Hurwitz condition \eqref{eq:RH}, we say that  $\delta$   is an {\em admissible factor} of $(g,b,d;\e)$ if
  \begin{itemize}
  \item $\delta$ is  a \emph{proper divisor} of $d$,  i.e.,  $\delta |d$  and $1<\delta<d$, 
  \item for all $i \in\{1,\ldots,n\}$, either $\delta|e_i$ or $e_i <\delta$,
  \item  $\displaystyle\sum_{\substack{\delta|e_i \\ 1<\delta<e_i}} \left(\frac{e_i}{\delta}-1\right)$ is even, and, if $b=0$, also $\geq 2\left(\frac{d}{\delta}-1\right)$.
  \end{itemize}
\end{defn}

  It can be easily verified that the last definition is independent of whether we augment $\e$ by adding some $2$'s corresponding to simple, unmarked ramification points.  We also note that the existence of an admissible factor implies $d \geq 4$.  One also sees that any proper divisor $\delta$ of $d$ such that $e_i<\delta$ for all $i$ is automatically admissible, as the sum in the last point is void.  However such a divisor may not exist.  
  
We will prove:

\begin{thm} \label{thm:composed}
  We have $\H^{\delta}_{g  \to b,d;\e}\neq \emptyset$  if and only if $\delta$ is an admissible factor of $(g,b,d;\e)$.
\end{thm}

This is an immediate consequence of Proposition \ref{prop:composed}  below,  which is a more precise version.

The natural question arises whether there exist  components of  {\it simple} type of  $\H_{g \to b,d;\e}$ in the numerical cases where there exist  components of  non--simple type. 

We give a positive answer in all cases when $b>1$:

\begin{thm} \label{thm:simple2}
For any  $(g,b,d;\e)$ satisfying the Riemann-Hurwitz condition \eqref{eq:RH} and $b \geq 2$, there always exists a component of $\H_{g  \to b,d;\e}$  of  simple    type.
\end{thm}

In the cases $b \leq 1$ we give a positive answer in  several  cases. For any
$(g,b,d;\e)$ satisfying the Riemann-Hurwitz condition and having an admissible factor, let $\Delta(g,b,d;\e)$ be the largest such. Consider the following conditions $(*)_b$, for $b=0,1$:

\[
   (*)_0:  b=0, \;\; n+r \geq 4, \;\; 
\#\{e_i: e_i \neq d\}+r \geq 2 \;\; \mbox{and} \;\; \displaystyle\sum_{e_i \neq d}(e_i-1)+r \geq \Delta(g,0,d;\e)
\]
 and
\[
 (*)_1:  b=1, \;\; n+r \geq 2, \;\;  \mbox{and} \;\;
2g-2 \geq \Delta(g,1,d;\e).
\]

 Then our result says:

 \begin{thm} \label{thm:simple01}
   Let $b \in \{0,1\}$ and $(g,b,d;\e)$  satisfy  the Riemann-Hurwitz condition \eqref{eq:RH} and  admit   an admissible factor. Assume that  $(*)_b$  holds.

   Then $\H_{g  \to b,d;\e}$ has a component of   simple    type.
\end{thm}

 \begin{remark} \label{rem:cond*}
   Note that condition $(*)_0$ is automatically satisfied if $n+r \geq 4$ and
    \linebreak $\#\{e_i: e_i = d\} \leq 1$ (see Lemma \ref{lemma:cond*}).
 \end{remark}

Theorems \ref{thm:simple2}--\ref{thm:simple01} will be proved in Section \ref{sec:scomp}. In Section \ref{sec:ex} we will consider some examples   regarding  admissible factors, condition $(*)_b$ and the existence of  components of simple and non-simple type. These examples show in particular that we cannot avoid the conditions $(*)_b$ completely. 

 We also remark that in the special case $b=0$, $n=2$, $\e=(d,d)$, that is, of covers of $\PP^1$ totally ramified at two points, a complete answer concerning the components of the Hurwitz space has recently been given in \cite[App.~A]{Agg}: in our language, $\H_{g  \to 0,d;(d,d)}$ has a unique irreducible  component of simple type and a unique irreducible component $\H^{\delta}_{g \to 0,d;(d,d)}$ for every proper nontrivial divisor $\delta$ of $d$, when $g \geq 1$ (see Example \ref{ex:d=2tot}).

\vspace{0.2cm} {\it Convention.} We work over the field of complex numbers.

\vspace{0.2cm} {\it Acknowledgments.} A.L.K. was partially supported 
	by the Trond Mohn Foundation's project ``Pure Mathematics in
	Norway''. C.C. and S.T. are members of GNSAGA of the Istituto Nazionale di Alta Matematica ``F. Severi''.

\section{Nonemptiness and dimension of $\H_{g \to b,d;\e}$} \label{sec:dimhur}

For the sake of the reader and later reference, we prove the following result in this section:

\begin{prop} \label{prop:dimhur}
  The Hurwitz moduli space $\H_{g \to b,d;\e}$ is
nonempty if and only if  the Riemann-Hurwitz condition
 \eqref{eq:RH} holds, in which case it is 
equidimensional of dimension 
  \begin{eqnarray} \label{eq:dimhur}
    \dim \H_{g \to b,d;\e}  & = & 2(g-1)-(2d-3)(b-1)+n-\displaystyle\sum_{i=1}^n(e_i-1)  \\
 \nonumber                           & = & 3(b-1)+n+r(g,b,d;\e)=3(g-1)+n+\sigma(g,b,d;\e),
\end{eqnarray}
except for the following cases:
  \begin{itemize}
    \item[(i)] $n=0$, $g=b=1$,
\item[(ii)]  $n=0$, $g=b=0$, $d=2$,
  \item[(iii)] $n=1$, $g=b=0$, $e_1=d=2$,
  \item[(iv)] $n=2$, $g=b=0$, $e_1=e_2=d$.
  \end{itemize}
  where $\dim \H_{g \to b,d;\e}=g$. Moreover, whenever $\H_{g \to b,d;\e}\neq \emptyset$, the forgetful map
  
  \begin{eqnarray} \label{eq:forgetful1}
  \H_{g \to b,d;\e} & \dashrightarrow & \M_{b,n} \\
 \nonumber  [\pi:(C,x_1,\ldots,x_n) \to B] & \mapsto & (B,\pi(x_1),\ldots,\pi(x_n))
  \end{eqnarray}
   is dominant.
\end{prop}

\begin{proof}   
  Fix  any $(B,y_1,\ldots,y_n) \in \M_{b,n}$ and consider the scheme $\H_{g \to b,d;\e}(B,y_1,\ldots,y_n)$
   of all  degree--$d$  covers  $\pi: C\to B$, with $C$ a smooth irreducible curve of genus $g$, and $\pi$ ramified at points $x_1,\ldots, x_n$ such that $\pi(x_i)=y_i$ for $1\leq i\leq n$, of orders $e_1,\ldots , e_n$. This is non--empty   if and only if the Riemann-Hurwitz condition
   \eqref{eq:RH} holds;   indeed, for $b>0$, this follows from \cite[Prop. 3.4]{EKS} (see also \cite[Thm.1.11]{Pe}), whereas for $b=0$, it follows from
 \cite[Cor. F]{CFG}. The dominance of \eqref{eq:forgetful1} follows.

 Let $r=r(g,b,d;\e):=2(g-1-d(b-1))-\displaystyle\sum_{i=1}^n(e_i-1)$  as defined in the introduction.  We have a forgetful rational map 
\[
  \xymatrix{
  \H_{g \to b,d;\e} \ar@{-->}[r] & \M_{b,n+r}/\Sym^r},
\]
where the symmetric group acts on the $r$ unmarked branch points of the base curve. By Riemann's Existence Theorem (cf., e.g., \cite[III, Cor. 4.9-10]{Mi}, this map is finite, whence
\begin{equation} \label{eq:dimhur_pf}
  \dim \H_{g \to b,d;\e} =\dim \M_{b,n+r}/\Sym^r = \dim \M_{b,n+r}.
  \end{equation}
Except for the special cases
\begin{equation}
  \label{eq:exceptions}
  (b,n+r) \in \{ (0,0),(0,1),(0,2),(1,0)\},
\end{equation}
we have
\begin{eqnarray*}
  \dim \M_{b,n+r} & = & 3b-3+n+r=3b-3+n+2(g-1-d(b-1))-\displaystyle\sum_{i=1}^n(e_i-1)  \\
  & = & 2(g-1)-(2d-3)(b-1)+n-\displaystyle\sum_{i=1}^n(e_i-1) ,
  \end{eqnarray*}
proving \eqref{eq:dimhur}. (The  last  equality in \eqref{eq:dimhur} is easily verified.)

We now treat the special cases in \eqref{eq:exceptions}.\medskip

{\bf Case $(b,n+r)=(0,0)$:} Then $r=2(g-1+d)=0$, whence $g=1-d<0$, a contradiction.\medskip

{\bf Case $(b,n+r)=(0,1)$:} Inserting $n=b=0$ in the expression for $r$ yields
$r=2(g-1+d) \neq 1$, a contradiction. Inserting $(b,n)=(0,1)$ in the expression for $r$ yields
$r=2(g-1+d)+1-e_1=0$, whence $e_1=2(g-1+d)+1 \geq 2(-1+d)+1=2d-1$, contradicting the fact that $e_1 \leq d$.\medskip

{\bf Case $(b,n+r)=(0,2)$:} 
We treat the cases $n \in\{0,1,2\}$ separately.

If $n=0$, then inserting $n=b=0$ in the expression for $r$ yields
$r=2(g-1+d)=2$, whence $g+d=2$, yielding the only possibility $(g,d)=(0,2)$. This is case (ii) in the statement.

If $n=1$, then inserting $(b,n)=(0,1)$ in the expression for $r$ yields
$r=2(g-1+d)+1-e_1=1$, whence $e_1= 2(g-1+d)$. Since $e_1 \leq d$, this yields the only possibility $g=0$ and $e_1=d=2$. This is case (iii). 

If $n=2$, then inserting $(b,n)=(0,2)$ in the expression for $r$ yields
$r=2(g-1+d)+2-e_1-e_2=0$, whence $e_1+e_2= 2(g+d)$. Since $e_i \leq d$,
the only possibility is $g=0$ and $e_1=e_2=d$, which is case (iv). 

In all these cases we compute $\dim \H_{0 \to 0,d;\e}=\dim \M_{0,2}=0$, as claimed.\medskip

{\bf Case $(b,n+r)=(1,0)$:} Inserting $n= r=0$  and $b=1$  in the expression for $r$ yields
$0=r=2(g-1)$, whence $g=1$, yielding case (i). We compute
$\dim \H_{1 \to 1,d;\mathbf{0}}=\dim \M_{1,0}=1$, as claimed.
\end{proof}

 \section{Components of the Hurwitz scheme of non--simple   type} \label{sec:nscomp}

 In this section we prove the following result, which is a more precise version of Theorem \ref{thm:composed}:

\begin{proposition} \label{prop:composed}
  There exists a component of $\H_{g  \to b,d;\e}$  of  non--simple    type  if and only if  $(g,b,d;\e)$ has an admissible factor $\delta$ (cf. Definition \ref{def:adm}). 

  For each  such $\delta$  there is at least one component of  $\H_{g  \to b,d;\e}$  whose general member  admits a factorization of the form \eqref{eq:compos}.  Moreover, denoting by $x_1,\ldots,x_n$  the ramification points of $\pi$, with $e_i$ being the order of $x_i$, and reordering indices such that $e_1,\ldots,e_m$ are precisely the $e_i$ with $\delta|e_i$ and $\delta<e_i$, the following conditions are satisfied: 
  
  \begin{itemize}
  \item[(i)] $C'$ has genus $g' :=1+\frac{d}{\delta}(b-1) +\frac{1}{2} \displaystyle\sum_{i=1}^m \left(\frac{e_i}{\delta}-1\right)$,
  \item[(ii)] 
    $\psi$  has degree  $\frac{d}{\delta}$  and  is ramified precisely  at  $\phi(x_1),\ldots,\phi(x_m)$, with  order  $\frac{e_i}{\delta}$  at $\phi(x_i)$,
  \item[(iii)] $\phi$  has degree  $\delta$  and  is totally ramified at $x_1,\ldots,x_m$, ramified at $x_j$ with order $e_j$  for each $j \in \{m+1,\ldots,n\}$, and otherwise simply ramified  at $r=r(g,b,d;\e)$ distinct points. 
  \end{itemize}
\end{proposition}

\begin{proof}
  Assume that a general member of a component of $\H_{g  \to b,d;\e}$ admits a composition as in \eqref{eq:compos}. Let $x_1,\ldots,x_n$ be the ramification points of $\pi$, with $e_i$ being the order of $x_i$. Let $\delta:=\deg \phi$,  so that $\deg \psi=\frac{d}{\delta}$.   As we

   remarked  in the  introduction,  by generality,    the points $\pi(x_1),\ldots,\pi(x_n)$ are distinct. Thus, also the points $y_i:=\phi(x_i) \in C'$ are distinct, for $i \in \{1,\ldots,n\}$. By generality, the ramification of $\pi$ outside the $x_i$ is simple and lie in other, distinct fibers of $\pi$. In particular the generality assumption implies that 
  \begin{equation}
    \label{eq:fifrapi}
    \pi^*(\pi(x_i))=e_ix_i+R_i \;\; \mbox{with $R_i$ effective, reduced and supported outside} \;\; \{x_1,\ldots,x_n\}.
  \end{equation}

  Let $\delta_i$ denote the order of ramification of $\phi$ at $x_i$, that is,
  \begin{equation}
    \label{eq:ramfi}
    \phi^*y_i=\delta_ix_i+R'_i, \;\; \mbox{with $R'_i$ effective, reduced and supported outside} \;\; \{x_1,\ldots,x_n\},
  \end{equation}
  for each $i \in \{1,\ldots,n\}$.

  \begin{claim} \label{cl1:prop:composed}
    If $\psi$ is ramified at $y_i$, then $\phi^*y_i=\delta x_i$, that is, $\phi$ is totally ramified at $x_i$.
  \end{claim}

  \begin{proof}[Proof of claim]
    If $\psi$ is ramified at $y_i$, we have $\psi^*\psi(y_i)\geq s_iy_i$ for some $s_i \geq 2$. Then
    \[ \pi^*\pi(x_i)=\phi^*(\psi^*\psi(y_i)) \geq s_i \phi^*y_i=s_i(\delta_ix_i+R'_i)=s_i\delta_ix_i+s_iR'_i,\]
    contradicting \eqref{eq:fifrapi}, unless $R'_i=0$, in which case
    $\phi^*y_i=\delta x_i$ by \eqref{eq:ramfi}. 
  \end{proof}

\begin{claim} \label{cl2:prop:composed}
The map $\psi$ is unramified outside $\{y_1,\ldots,y_n\}$.
\end{claim}

  \begin{proof}[Proof of claim]
    If $\psi$ is ramified at some point $z \not \in \{y_1,\ldots,y_n\}$, we have $\psi^*\psi(z)\geq sz$ for some $s \geq 2$. Then
    \[ \pi^*\psi(z)=\phi^*(\psi^*\psi(z)) \geq s \phi^*z,\]
    contradicting the fact that the ramification of $\pi$ outside the $x_i$ is simple and lie in other, distinct fibers of $\pi$.
  \end{proof}

  By the claims we may reorder the $x_i$ so that $y_1,\ldots,y_m$ are precisely the ramification points of $\psi$, for some integer $m \leq n$, and then $\phi$ is  totally ramified (thus of order $\delta$) at $x_1,\ldots,x_m$ and consequently the ramification order of $\psi$ at each $y_i$ is $\frac{e_i}{\delta}$ for $i \in \{1,\ldots,m\}$.
  
  For $j \in \{m+1,\ldots,n\}$ the ramification order of $\phi$ at $x_j$ is $e_j$, whence $e_i<\delta$,  and the generality of $\pi$ implies that $\phi$ is otherwise simply ramified. By  the Riemann-Hurwitz  formula on $\psi$, the genus $g'$ of $C'$ satisfies
  \[ 2(g'-1)=2\frac{d}{\delta}(b-1)+\displaystyle\sum_{i=1}^m\left(\frac{e_i}{\delta}-1\right),\]
  which can be rewritten as 
  \begin{equation}
    \label{eq:g'Hur}
g':=1+\frac{d}{\delta}(b-1) +\frac{1}{2} \displaystyle\sum_{i=1}^m \left(\frac{e_i}{\delta}-1\right),
  \end{equation}
  which in particular must be a natural number.  Consequently,
$ \displaystyle\sum_{i=1}^m \left(\frac{e_i}{\delta}-1\right)$ is even, and, if $b=0$, also $\geq 2\left(\frac{d}{\delta}-1\right)$. Hence, $\delta$ is admissible factor of $(g,b,d,\e)$.  

We have thus concluded the ``only if''-part of the proposition. 

Conversely, assume the existence of the integer $\delta$  as in the hypotheses.  Then $g'$ as defined in \eqref{eq:g'Hur} is a natural number.  Set $m:=\#\{ e_i : \delta|e_i \;  \mbox{and} \; \delta<e_i \}$.  By  Proposition  \ref{prop:dimhur},
for any $m$-pointed curve $(B,z_1,\ldots,z_m) \in \M_{b,m}$,   there is a  smooth  irreducible   projective  curve $C'$ of genus $g'$ with a degree-$\frac{d}{\delta}$ cover to
 $B$  ramified precisely at $m$ points $y_1,\ldots,y_m$  over $z_1,\ldots,z_m$,  with order $\frac{e_i}{\delta}$, for $i \in \{1,\ldots,m\}$.
Then,  again  by
 Proposition  \ref{prop:dimhur},    there is a  smooth   irreducible   projective  curve $C$ with a degree-$\delta$ cover to $C'$ totally branched at $y_1,\ldots,y_m$, and branched at $n-m$ further points with order $e_{m+1},\ldots,e_n$ and finally simply branched at  $r=r(g,b,d;\e)$  points.

  Using the Riemann-Hurwitz   formula on $\phi$  and \eqref{eq:g'Hur}, one verifies that $g(C)=g$.  Thus $\pi:=\psi \circ \phi: C \to \PP^1$ is a member of
$\H_{g  \to b,d;\e}$, as desired.

We have to prove that this construction yields entire  components  of
$\H_{g  \to b,d;\e}$. To do so, we compute the number of parameters  of the construction. 

The construction of the map $\psi:C' \to B$ depends on $ \dim \H_{g'  \to b,\frac{d}{\delta};\e'}$ parameters, where $\e':=\left(\frac{e_1}{\delta},\ldots, \frac{e_m}{\delta}\right)$.  We can compute this number from  Proposition \ref{prop:dimhur}. If we are in one of the exceptional cases therein, then two cases may occur:
\begin{itemize}
\item $g'=b=0$ and $m \in \{0,1,2\}$. From \eqref{eq:g'Hur} we must have $m=2$, whence $\frac{e_1}{\delta}=\frac{e_2}{\delta}=\frac{d}{\delta}$ by Proposition \ref{prop:dimhur}, and $\dim \H_{ 0 \to 0,\frac{d}{\delta};\e'}=0$. The fact that $\frac{e_1}{\delta}=\frac{e_2}{\delta}=\frac{d}{\delta}$ is also a consequence of \eqref{eq:g'Hur}.
  \item $g'=b=1$, $m=0$. Then $\dim \H_{ 1 \to 1,\frac{d}{\delta};\mathbf{0}}=1$ by Proposition \ref{prop:dimhur}.
\end{itemize}
In all other cases, we have, by Proposition \ref{prop:dimhur}  and Claim \ref{cl1:prop:composed},
\[
  \dim  \H_{g' \to b,\frac{d}{\delta};\e'}  =  3(b-1)+m.
\]
To summarize, the construction of $\psi$ depends on
the following number of parameters:
\begin{equation}
  \label{eq:parpsi}
  3(b-1)+m+\epsilon, \;\; \mbox{with} \;\; \epsilon= \begin{cases}
    1, & \mbox{if} \;\; (g',b,m) \in \{(0,0,2),(1,1,0)\},\\
    0, & \mbox{otherwise},
  \end{cases}
\end{equation}

The map $\phi:C \to C'$ is now forced to be totally ramified at  the $m$ points on $C'$ that are ramification points of $\psi$, so the construction of $\phi$ depends on the choice of the remaining $n-m$ ramification points completing the sequence of $n$ prescribed ramification points of $\pi$, plus the  
$r=r(g,b,d;\e)$   simple ramification points, that is, on
\[ (n-m)+r=(n-m)+2(g-1-d(b-1))-\displaystyle\sum_{i=1}^n\left(e_i-1\right) =2(g-1-d(b-1)+n)-m-\displaystyle\sum_{i=1}^ne_i \]
parameters. On the other hand,
we must also subtract the  dimension of the family  of automorphisms of $C'$ fixing the $m$ ramification points of $\psi$, which can only be positive when  $g'=0$ and $m \leq 2$  or $g'=1$ and $m=0$.  In these cases, \eqref{eq:g'Hur} implies  $(b,m)=(0,2)$  or $b=1$, respectively  and the dimension of the family of automorphisms of $C'$ fixing the $m$ ramification points of $\psi$ is 1.   
Thus, the number of parameters on which $\phi$ depends is
\begin{equation}
  \label{eq:parfi}
2(g-1-d(b-1)+n)-m-\displaystyle\sum_{i=1}^ne_i -\epsilon, \;\;\mbox{with $\epsilon$ as in \eqref{eq:parpsi}}.
\end{equation}
Summing up, the construction of $\pi$ via $\psi$ and $\phi$ depends on the number of parameters given by adding \eqref{eq:parpsi} and \eqref{eq:parfi},  which is,
\begin{eqnarray*}
   & & 3(b-1)+m + 2(g-1-d(b-1)+n)-m-\displaystyle\sum_{i=1}^ne_i   \\
  & =  & 2(g-1)-(2d-3)(b-1)+n-\displaystyle\sum_{i=1}^n\left(e_i-1\right)   =    \dim \H_{g  \to b,d;\e},
\end{eqnarray*}
by  Proposition \ref{prop:dimhur}. 
\end{proof}

\begin{remark} \label{rem:composed}
  It follows from the proof that for the general member
  $[\pi: (C,x_1,\ldots,x_n) \to  B]$ in a component of $\H_{g  \to b ,d;\e}$  of  non--simple    type,  then $[\psi:(C',\phi(x_1),\ldots,\phi(x_m)) \to B]$  is a general member of $\H_{g' \to b,\frac{d}{\delta},\e'}$, with $\e':=(\frac{e_1}{\delta},\ldots,\frac{e_m}{\delta})$. Moreover, 
  \begin{eqnarray*}
    \sigma\left(g',b,\frac{d}{\delta};\e'\right) & = & m-\frac{1}{\delta}\displaystyle\sum_{i=1}^me_i-(2\frac{d}{\delta}-3)(b-1)-g'+1 \\
 & \stackrel{\eqref{eq:g'Hur}}{=} &    -2(g'-1)+3(b-1)-g'+1=-3(g'-b),
  \end{eqnarray*}
is the mimimal value. This is also a consequence of the fact that all ramification of $\psi$ is determined. 
\end{remark}

\begin{remark} \label{rem:composed2}
  Note that if $b=0$, the  last bullet point in Definition \ref{def:adm}  requires  $m:= \#\{i\;|\; \delta|e_i \; \mbox{and} \; \delta<e_i\} \geq 2$.  
\end{remark}

We state, for later use, the following corollaries in two special cases:

\begin{cor} \label{cor:claimC}
  Let  $[\pi: C\to B] \in \H_{g \to b,d;\e}^{\delta}$.    Then
  $e_i \neq d-1$ for all $i$, and if $e_i=d-2$ for some $i$, then $\delta=2$.
\end{cor}

\begin{proof}
  By Proposition \ref{prop:composed} each $e_i$ is either divisible by $\delta$ or $<\delta$. It follows that $e_i \neq d-1$. Moreover, if $e_i=d-2$, then  there are two possibilities:
  \begin{itemize}
  \item $\delta|(d-2)$; then,  since $\delta|d$, we must have $\delta=2$.
  \item $d-2 <\delta$; this is equivalent to 
    $\delta(\frac{d}{\delta}-1)<2$, an absurdity.
    \end{itemize}
  \end{proof}

\begin{cor} \label{cor:claimD}
  Let   $[\pi: C\to B] \in \H_{g \to 0,d;\e}^{\delta}$ be    such that
 $n \geq 3$ and $r(g,0,d;\e)=0$ (that is, all ramification points are marked and they are at least three). 
Assume that there is some point of total ramification, say $e_1=d$.

Then $\displaystyle\sum_{i=2}^{n}(e_i-1) \geq d$.
  \end{cor}

\begin{proof}
  By Remark \ref{rem:composed2}, we have
  $m:=\#\{ e_i : \delta|e_i\} \geq 2$,
  and the ramification point of order $d$ must be among the $m$ points where $\phi$ is totally ramified.

  Inserting into the formula for $g'$ in Proposition \ref{prop:composed}, we obtain, after relabelling some indices,
  \begin{eqnarray*}
    g' & = & 1-\frac{d}{\delta}+\frac{1}{2}\left[\left(\frac{d}{\delta}-1\right)
+\displaystyle\sum_{i=2}^{m}\left(\frac{e_i}{\delta}-1\right)\right]\\
    & = & \frac{1}{2\delta}\left[2\delta-2d+d-\delta+ \displaystyle\sum_{i=2}^{m}\left(e_i-\delta\right)\right]
\\
    & = & \frac{1}{2\delta}\left[-d-(\delta-1)(m-2)+1+\displaystyle\sum_{i=2}^{m}(e_i-1)\right],
  \end{eqnarray*}
  whence
  \[ \displaystyle\sum_{i=2}^{m}(e_i-1) \geq d+(\delta-1)(m-2)-1.\]
  If $m=2$, we have $n>m$, whence
  \[ \displaystyle\sum_{i=2}^{n}(e_i-1) \geq \displaystyle\sum_{i=2}^{m}(e_i-1) +1 \geq (d-1)+1=d,\]
  as desired.
  If $m>2$, then 
 \[ \displaystyle\sum_{i=2}^{m}(e_i-1) \geq d+(\delta-1)-1=d+\delta-2\geq d.\]
  \end{proof}

 \section{Components of the Hurwitz scheme of simple   type} \label{sec:scomp}

In this section we will prove Theorems \ref{thm:simple2}--\ref{thm:simple01}.
 We will prove  them  by constructing, for any $(g,b,d;\e)$ satisfying the conditions of Proposition \ref{prop:composed}, a member of the border of $\H_{g \to b,d;\e}$ not lying in the border of
 $\H^{\delta}_{g \to b,d;\e}$ for any  admissible factor $\delta$.

\subsection{ Preliminaries} To simplify notation we will mark all remaining simple ramification points, thus setting $r(g,b,d;\e)=0$. Conditions $(*)_0$ now read
\begin{equation}
  \label{eq:st0}  b=0, \;\; n \geq 4, \;\; \#\{e_i: e_i \neq d\} \geq 2 \;\; \mbox{and} \;\; \displaystyle\sum_{e_i \neq d}(e_i-1) \geq \Delta(g,0,d;\e),
 \end{equation}
where we recall that $\Delta(g,b,d;\e)$ is the largest admissible factor of
$(g,b,d;\e)$.
Conditions $(*)_1$ read
\begin{equation}
 \label{eq:st1'} b=1, \;\; n \geq 2, \;\; \;\; \mbox{and} 
     \;\; 2g-2 \geq \Delta(g,1,d;\e).
 \end{equation}

We will need to make a couple of observations regarding these conditions. The first one proves what we claimed in Remark \ref{rem:cond*} in the introduction:

\begin{lemma} \label{lemma:cond*}
  For $b=0$, condition \eqref{eq:st0} is automatically satisfied if $n \geq 4$
  and \linebreak $\#\{e_i: e_i = d\} \leq 1$. More precisely, in this case we may find $i_1 \neq i_2$ such that $e_{i_1},e_{i_2} \neq d$ and 
\begin{equation} \label{eq:cond*}
 \Delta(g,0,d;\e) \leq  e_{i_1}+e_{i_2}-2 \leq \displaystyle\sum_{i \neq i_1,i_2}(e_i-1). 
\end{equation}
\end{lemma}

\begin{proof}
  Set for simplicity $\delta=\Delta(g,0,d;\e)$. Recall that at least two of the  $e_i$ are divisible by $\delta$, by Remark \ref{rem:composed2}. Therefore, at least one $e_i$ satisfies $e_i \neq d$ and $\delta|e_i$.  We may then define
  \begin{eqnarray*}
    e_{i_1} & := & \min\{e_i\;:\; \delta|e_i \;\mbox{and}\; e_i \neq d\}, \\
    e_{i_2} & := & \min\{e_i\;:\; i\neq i_1\},
  \end{eqnarray*}
  and both $e_{i_1}, e_{i_2} \neq d$. By choice of $e_{i_1}$,  since $b=0$ and therefore $m\geq 2$ (see Remark \ref {rem:composed2}),  there exists some $i_3 \not \in \{i_1,i_2\}$ such that
  $\delta|e_{i_3}$ and $e_{i_3} \geq e_{i_1}$.  
  Since $n \geq 4$, there exists, by choice of $e_{i_2}$, some $i_4 \not \in \{i_1,i_2,i_3\}$ such that $ e_{i_4} \geq e_{i_2}$. Therefore,
  \[ e_{i_1}+e_{i_2}-2 \leq (e_{ i_3}-1)+(e_{i_4}-1) \leq \displaystyle\sum_{i\neq i_1,i_2}(e_i-1).\]
   This shows that the right hand inequality of \eqref{eq:cond*} is satisfied. 
  Finally, since $\delta|e_{i_1}$, we may write $e_{i_1}=a\delta$ for some integer $a \geq 1$.  Then
  \[ e_{i_1}+ e_{i_2}= a\delta+e_{i_2} \geq a\delta+2 \geq \delta+2,\]
   which proves the left hand inequality of \eqref{eq:cond*}.
\end{proof}

Regarding \eqref{eq:st1'}, we note that if it is satisfied with $2g-2<d$, then, since by the Riemann-Hurwitz formula $2g-2=\displaystyle\sum_{i=1}^n(e_i-1)$, we must have $e_i \neq d$ for all $i$. In particular, we may (and will find it convenient to)  split \eqref{eq:st1'}  into the two cases:
\begin{equation}
 \label{eq:st1} b=1, \;\; \#\{e_i: e_i \neq d\} \geq 2 \;\; \mbox{and} \;\; \displaystyle\sum_{e_i \neq d}(e_i-1) \geq \Delta(g,1,d;\e).
\end{equation}
or
\begin{equation}
 \label{eq:st2} b=1, \;\; n \geq 2, \;\; \;\; \mbox{and} 
     \;\; 2g-2 \geq d.
 \end{equation}

We will now separate the proofs of Theorems \ref{thm:simple2}--\ref{thm:simple01} into two parts. Firstly, we assume $b \in \{0,1\}$ and prove parts of Theorem \ref{thm:simple01} under the assumption \eqref{eq:st0} for $b=0$ and \eqref{eq:st1} for $b=1$. Then we prove the remaining part of Theorem \ref{thm:simple01}, which is the case $b=1$ under assumption \eqref{eq:st2}, and Theorem \ref{thm:simple2} (where $b \geq 2$).

\subsection{The proof of Theorem \ref{thm:simple01} under assumption
  \eqref{eq:st0}
for $b=0$ and \eqref{eq:st1} for $b=1$}

We start with a couple of helpful results.

 \begin{lemma} \label{lemma:firstred}
   Assume that $b=0$ (respectively, $b=1$) and \eqref{eq:st0} (resp., \eqref{eq:st1}) holds. Then, there is an integer $s$ with $2 \leq s \leq n-2$ (resp., $2 \leq s$) such that, after relabelling the $e_i$s, conditions (i)-(iii) (resp., (i)-(ii)) below are satisfied: 
   \begin{itemize}
   \item[(i)] $e_i \neq d$ for $i \in \{1,\ldots,s\}$,
   \item[(ii)] $\displaystyle\sum_{i=1}^s(e_i-1) \geq \Delta(g,b,d;\e)$,
   \item[(iii)] either \; $\displaystyle\sum_{i=1}^s(e_i-1) \leq \displaystyle\sum_{i=s+1}^n(e_i-1)$ \; or \; $2d-2 \leq \displaystyle\sum_{i=s+1}^n(e_i-1)$.      
   \end{itemize}
 \end{lemma}

 \begin{proof}
   If $b=1$, then the result is an immediate consequence of the conditions in \eqref{eq:st1}. We therefore assume that $b=0$ in the rest of the proof.

 If $\#\{e_i: e_i = d\} \leq 1$, then the result follows from Lemma \ref{lemma:cond*}, with the left hand option in (iii) being fulfilled. 

  Assume finally that $\#\{e_i: e_i = d\} \geq 2$. We label all $e_i$ such that $e_i \neq d$ by $e_1,\ldots, e_s$. Then \eqref{eq:st0} yields $2 \leq s \leq n-2$ and (i)-(ii). Moreover, as
 $\#\{e_i: e_i = d\} \geq 2$, we have $\displaystyle\sum_{i=s+1}^n(e_i-1)\geq 2(d-1)$, which is the right hand alternative of (iii).  
 \end{proof}

 We  will   need the following additional property:

 \begin{lemma} \label{lemma:firstred_min}
   Assume that $s$ and $e_1,\ldots,e_s$ are as in Lemma \ref{lemma:firstred}, with $s$ {\it minimal}.

   If $\displaystyle\sum_{i=1}^s(e_i-1)\geq d$, then $s=2$.
 \end{lemma}

 \begin{proof}
   Assume that $s \geq 3$. Any  subset  of $s-1$ elements of $\{e_1,\ldots,e_s\}$ still satisfy conditions (i) and (iii) of Lemma \ref{lemma:firstred}. The minimality of $s$ thus implies that for any $e_* \in \{e_1,\ldots,e_s\}$, we have $\displaystyle\sum_{i=1}^s(e_i-1) -(e_*-1) \leq \Delta(g,b,d;\e)-1$. Therefore,
   \[ e_* \geq \displaystyle\sum_{i=1}^s(e_i-1)-\Delta(g,b,d;\e)+2 \geq
   d-\Delta(g,b,d;\e)+2 \geq d-\frac{d}{2}+2=\frac{d}{2}+2,\]
   where we have used the fact that $\Delta(g,b,d;\e) \leq \frac{d}{2}$.
   In particular,
   \[ (e_1-1)+(e_2-1) \geq 2\left(\frac{d}{2}+1\right)=d+2>\Delta(g,b,d;\e),\]
   contradicting the minimality of $s$.
 \end{proof}

Pick now $e_1,\ldots,e_s$ as in Lemma \ref{lemma:firstred}, with $s$ {\it minimal}. We  define:

\begin{equation} \label{eq:defalfa}
  \alpha:=\begin{cases}

  \displaystyle\sum_{i=1}^s(e_i-1) +1, & \mbox{if} \;\; d \geq \displaystyle\sum_{i=1}^s(e_i-1)+1, \\
  d-1, & \mbox{if} \;\; d \leq  \displaystyle\sum_{i=1}^s(e_i-1) \;\; \mbox{and} \;\; \displaystyle\sum_{i=1}^s(e_i-1) +d \;\; \mbox{is even}, \\
  d-2, & \mbox{if} \;\; d \leq  \displaystyle\sum_{i=1}^s(e_i-1) \;\; \mbox{and} \;\; \displaystyle\sum_{i=1}^s(e_i-1) +d \;\; \mbox{is odd}.
  \end{cases}
\end{equation}

\begin{remark} \label{rem:minimals}
In the two lower cases of \eqref{eq:defalfa}, we have $s=2$ by Lemma \ref{lemma:firstred_min}.
 \end{remark}

\begin{lemma} \label{lemma:condalpha}
  The integer $\alpha$ in \eqref{eq:defalfa} satisfies
  \begin{eqnarray}
\label{eq:A7} & \displaystyle\sum_{i=1}^s(e_i-1)+\alpha \;\; \mbox{is odd}, & \\
    \label{eq:A5} & \max\{e_1,\ldots,e_s\} \leq \alpha \leq d, & \\
        \label{eq:A100}     & \alpha>\Delta(g,b,d;\e), &  \\
        \label{eq:A80} & \alpha \leq \displaystyle\sum_{i=1}^s(e_i-1)+1, & \\
\label{eq:A8} &  \alpha \geq 2d-\displaystyle\sum_{i=s+1}^n(e_i-1)-1, \;\;\mbox{if} \;\; b=0.
\end{eqnarray}
\end{lemma}

\begin{proof}
  We treat the three cases in \eqref{eq:defalfa} separately.

    {\bf Case I: $\alpha=\displaystyle\sum_{i=1}^s(e_i-1) +1$ and $d \geq  \displaystyle\sum_{i=1}^s(e_i-1)+1$.}

  Properties \eqref{eq:A7}, \eqref{eq:A5} and \eqref{eq:A80} are obviously satisfied. Moreover, \eqref{eq:A100} is satisfied by Lemma \ref{lemma:firstred}(ii). 

   When $b=0$,  Riemann-Hurwitz yields 
  \[ \displaystyle\sum_{i=1}^s(e_i-1)+\displaystyle\sum_{i=s+1}^n(e_i-1) -2d=2(g-1) \geq -2, \]
  whence
  \[ 2d-\displaystyle\sum_{i=s+1}^n(e_i-1) \leq \displaystyle\sum_{i=1}^s(e_i-1)+2=\alpha +1,\]
  whence  \eqref{eq:A8} is satisfied.

  {\bf Case II: $\alpha=d-1$, $d \leq  \displaystyle\sum_{i=1}^s(e_i-1)$ and $\displaystyle\sum_{i=1}^s(e_i-1) +d$ is even.}

   Properties \eqref{eq:A7} and \eqref{eq:A80} are  automatically satisfied. Corollary \ref{cor:claimC}  yields that \eqref{eq:A5} is satisfied.  Condition \eqref{eq:A100} is clearly satisfied, as any admissible factor is $<d$.

  We have left to prove \eqref{eq:A8}, when $b=0$. We consider the two possibilities in Lemma \ref{lemma:firstred}(iii):
 \begin{itemize}
 \item If $\displaystyle\sum_{i=1}^s(e_i-1) \leq \displaystyle\sum_{i=s+1}^n(e_i-1)$, then, by the assumption that $d \leq  \displaystyle\sum_{i=1}^s(e_i-1)$ in this case, we have $2d-\displaystyle\sum_{i=s+1}^n(e_i-1)\leq 2d-\displaystyle\sum_{i=1}^s(e_i-1) \leq 2d-d=d=\alpha+1$.
 \item If $2d-2 \leq \displaystyle\sum_{i=s+1}^n(e_i-1)$, then
 $2d-\displaystyle\sum_{i=s+1}^n(e_i-1)\leq 2<d=\alpha+1$. 
 \end{itemize}
 Hence, \eqref{eq:A8} is satisfied.

 {\bf Case III: $\alpha=d-2$, $d \leq  \displaystyle\sum_{i=1}^s(e_i-1)$ and $\displaystyle\sum_{i=1}^s(e_i-1) +d$ is odd.}

 Properties \eqref{eq:A7} and \eqref{eq:A80} are automatically satisfied. Lemma \ref{lemma:firstred}(i)  and Corollary \ref{cor:claimC} imply that $e_1,\ldots,e_s \leq d-2$, whence \eqref{eq:A5} is satisfied.

  If \eqref{eq:A100} is not satisfied, then $d-2 \leq \Delta(g,b,d;\e)$, and since the latter is a proper divisor of $d$, we must have $\Delta(g,b,d;\e)=2$, whence $d$ and all $e_i$  with $i\leq s$  are even. But $s=2$ by Remark \ref{rem:minimals}, whence $\displaystyle\sum_{i=1}^s(e_i-1) +d$  is even, contradicting the assumptions of this case.

  We have left to prove \eqref{eq:A8}, when $b=0$. We consider the two possibilities in Lemma \ref{lemma:firstred}(iii):
 \begin{itemize}
 \item If $\displaystyle\sum_{i=1}^s(e_i-1) \leq \displaystyle\sum_{i=s+1}^n(e_i-1)$, then, by the assumption that $d \leq  \displaystyle\sum_{i=1}^s(e_i-1)$ in this case, we have $2d-\displaystyle\sum_{i=s+1}^n(e_i-1)\leq 2d-\displaystyle\sum_{i=1}^s(e_i-1) \leq 2d-d=d=\alpha+2$. If equality holds, then
   $\displaystyle\sum_{i=1}^s(e_i-1)=d$, whence  $\displaystyle\sum_{i=1}^s(e_i-1)+d=2d$ is even, a contradiction on the assumption in this case. Hence,
$2d-\displaystyle\sum_{i=s+1}^n(e_i-1) \leq \alpha+1$, which is \eqref{eq:A8}.  
\item If $2d-2 \leq \displaystyle\sum_{i=s+1}^n(e_i-1)$, then
 $2d-\displaystyle\sum_{i=s+1}^n(e_i-1)\leq 2<d-1=\alpha+1$, which yields \eqref{eq:A8}. 
 \end{itemize}
 
 We conclude that \eqref{eq:A8} is satisfied. 
  \end{proof}

 Set
\begin{eqnarray*}
  g_1 & = & \frac{\displaystyle\sum_{i=1}^s(e_i-1)-\alpha+1}{2}, \\
  g_2 & = & d(b-1)+\frac{\displaystyle\sum_{i=s+1}^n(e_i-1)+\alpha+1}{2}.
\end{eqnarray*}
Condition \eqref{eq:A7} guarantees that $g_1 \in \ZZ$ and \eqref{eq:A80} guarantees that $g_1 \geq 0$. Condition \eqref{eq:A7} together with the fact that $\displaystyle\sum_{i=1}^n(e_i-1)$ is even by the Riemann-Hurwitz formula, guarantees that $\displaystyle\sum_{i=s+1}^n(e_i-1)+\alpha$ is odd as well, so that $g_2 \in \ZZ$.  Condition \eqref{eq:A8}  guarantees that $g_2 \geq 0$. Condition \eqref{eq:A5} and Proposition \ref{prop:dimhur} now  yield  the existence of two covers
\[
  [f_1:C_1 \to B_1 \cong \PP^1]   \in \H_{g_1\to 0,\alpha;e_1,\ldots,e_s,\alpha}\]
  and
  \[
    [f_2:C_2 \to B_2] \in  \H_{g_2\to b,d;\alpha,e_{s+1}, \ldots,e_n}.\] 
Both covers have a  marked  ramification point of order $\alpha$. We glue the two curves $C_1$ and $C_2$ along this point and call it $p$. We have
\[ f_2^*(f_2(p))=\alpha p+p_1+\cdots+p_{d-\alpha} \;\; \mbox{on} \;\; C_2.\]
Attach one copy of $B_1 \cong \PP^1$ to $C_2$ at each $p_i$, so that it intersects no other copy of $B_1$ and is disjoint from $C_1$. Define
\begin{eqnarray*}
  X_1 & := & C_1+\mbox{the $(d-\alpha)$ disjoint copies of $B_1$}, \\
  X   & := & X_1+C_2, 
\end{eqnarray*}
so that $X$ looks like this, where we have marked the  $s$  ramification points on $C_1$ other than $p$ by $x_1,\ldots,x_s$,  and the  $n-s$  ramification points on $C_2$ other than $p$ by $ x_{s+1},\ldots,x_n$:

\begin{center}
\begin{tikzpicture}[scale=0.85]

\draw [thick,red] (3.5,-4.5) .. controls (5,-2) and (6,-1) .. (9,1)
node [pos=0.65, circle, fill, inner sep=1.5pt] {}
        node [pos=0.77, circle, fill, inner sep=1.5pt] {}
        node [pos=0.83, circle, fill, inner sep=1pt] {}
        node [pos=0.865, circle, fill, inner sep=1pt] {}
        node [pos=0.9, circle, fill, inner sep=1pt] {}
        node [pos=0.95, circle, fill, inner sep=1.5pt] {};

\node at (3.95,-4.25) {\small $p$};
\node at (4.75,-3.15) {\small $p_1$};
\node at (5.15,-2.6) {\small $p_2$};
\node at (6.77,-1.2) {\small $p_{d-\alpha}$};
\node[red] at (7.03,-0.85) {\small $x_{s+1}$};
\node[red] at (7.7,-0.4) {\small $x_{s+2}$};
\node[red] at (8.88,0.6) {\small $x_{n}$};
\node[red] at (8.7,1.2) {\small $C_2$};

  \draw [thick,blue,xshift=-2.2,yshift=-18] (0,0) .. controls (2,-1) and (3,-2) .. (4,-4)
        node [pos=0.3, circle, fill, inner sep=1.5pt] {}
        node [pos=0.4, circle, fill, inner sep=1.5pt] {}
        node [pos=0.5, circle, fill, inner sep=1pt] {}
        node [pos=0.55, circle, fill, inner sep=1pt] {}
        node [pos=0.6, circle, fill, inner sep=1pt] {}
        node [pos=0.7, circle, fill, inner sep=1.5pt] {}
        node [pos=0.8, circle, fill, inner sep=1.5pt] {}
        node [pos=0.925, circle, fill=black, inner sep=1.5pt] {};

\node[blue] at (3.2,-3.8) {\small $x_1$};        
\node[blue] at (2.9,-3.3) {\small $x_2$};        
\node[blue] at (1.6,-2.2) {\small $x_{s-1}$};        
\node[blue] at (1.2,-1.6) {\small $x_s$};        
\node[blue] at (0,-1) {\small $C_1$};        

\draw[blue,stealth-stealth] (1.5,0)  -- node[left]{\small $(d-\alpha)B_1$} (3.2,2.2);

\draw [thick,blue,xshift=19,yshift=12] (1.3,-0.7) .. controls (2,-1) and (3,-2) .. (4,-4)
        node [pos=0.3, circle, fill, inner sep=1.5pt] {}
        node [pos=0.4, circle, fill, inner sep=1.5pt] {}
        node [pos=0.5, circle, fill, inner sep=1pt] {}
        node [pos=0.55, circle, fill, inner sep=1pt] {}
        node [pos=0.6, circle, fill, inner sep=1pt] {}
        node [pos=0.7, circle, fill, inner sep=1.5pt] {}
        node [pos=0.8, circle, fill, inner sep=1.5pt] {}
        node [pos=0.915, circle, fill=black, inner sep=1.5pt] {};

\draw [thick,blue,xshift=30,yshift=29] (1.3,-0.7) .. controls (2,-1) and (3,-2) .. (4,-4)
        node [pos=0.3, circle, fill, inner sep=1.5pt] {}
        node [pos=0.4, circle, fill, inner sep=1.5pt] {}
        node [pos=0.5, circle, fill, inner sep=1pt] {}
        node [pos=0.55, circle, fill, inner sep=1pt] {}
        node [pos=0.6, circle, fill, inner sep=1pt] {}
        node [pos=0.7, circle, fill, inner sep=1.5pt] {}
        node [pos=0.8, circle, fill, inner sep=1.5pt] {}
        node [pos=0.925, circle, fill=black, inner sep=1.5pt] {};

\draw [thick,blue,xshift=70,yshift=70] (1.3,-0.7) .. controls (2,-1) and (3,-2) .. (4,-4)
        node [pos=0.3, circle, fill, inner sep=1.5pt] {}
        node [pos=0.4, circle, fill, inner sep=1.5pt] {}
        node [pos=0.5, circle, fill, inner sep=1pt] {}
        node [pos=0.55, circle, fill, inner sep=1pt] {}
        node [pos=0.6, circle, fill, inner sep=1pt] {}
        node [pos=0.7, circle, fill, inner sep=1.5pt] {}
        node [pos=0.8, circle, fill, inner sep=1.5pt] {}
        node [pos=0.915, circle, fill=black, inner sep=1.5pt] {};

\draw [white,xshift=55,yshift=55] (1.3,-0.7) .. controls (2,-1) and (3,-2) .. (4,-4)
node [pos=0.5, circle, fill=blue, inner sep=0.7pt] {}
        node [pos=0.55, circle, fill=blue, inner sep=0.7pt] {}
        node [pos=0.6, circle, fill=blue, inner sep=0.7pt] {}
        node [pos=0.915, circle, fill=black, inner sep=1pt] {};

\draw [white,xshift=50,yshift=50] (1.3,-0.7) .. controls (2,-1) and (3,-2) .. (4,-4)
node [pos=0.5, circle, fill=blue, inner sep=0.7pt] {}
        node [pos=0.55, circle, fill=blue, inner sep=0.7pt] {}
        node [pos=0.6, circle, fill=blue, inner sep=0.7pt] {}
        node [pos=0.915, circle, fill=black, inner sep=1pt] {};
                
 \draw [white,xshift=45,yshift=45] (1.3,-0.7) .. controls (2,-1) and (3,-2) .. (4,-4)
node [pos=0.5, circle, fill=blue, inner sep=0.7pt] {}
        node [pos=0.55, circle, fill=blue, inner sep=0.7pt] {}
        node [pos=0.6, circle, fill=blue, inner sep=0.7pt] {}
           node [pos=0.915, circle, fill=black, inner sep=1pt] {};
\end{tikzpicture}

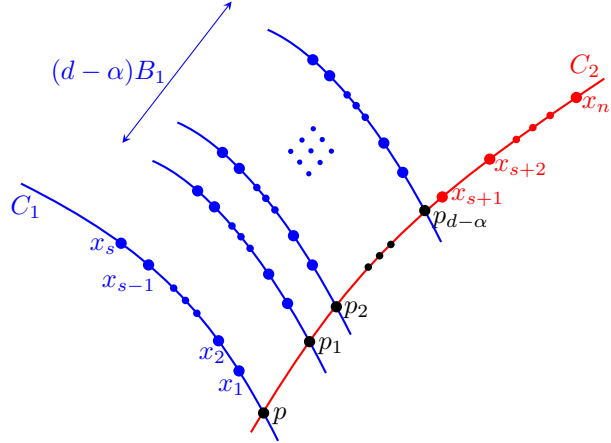
\tiny{\textsc{The curve $X=C_1 \cup C_2 \cup [\mbox{$d-\alpha$ copies of $B_1$}]$}}
\end{center}

  \noindent The  marked  points on each copy of $B_1$ are $f_1(x_1),\ldots,f_1(x_s)$. One computes that $p_a(X)=g_1+g_2=g$. 

  Now we glue $B_1$ and $B_2$ along $f_1(p) \in B_1$ and $f_2(p)\in B_2$, and define $B:=B_1 \cup B_2$. Then $p_a(B)=b$, We can then define a map $f:X \to B$ by gluing $f_1$ and $f_2$ in the following way:
  \begin{eqnarray*}
    f|_{C_1} & = & f_1, \\
    f|_{C_2} & = & f_2, \\
    f \; & \mbox{on} & \mbox{each copy of $B_1$ is the identity},
  \end{eqnarray*}
  as shown in the following picture.

\begin{center}
\begin{tikzpicture}[scale=0.85]

\draw [thick,red] (3.5,-4.5) .. controls (5,-2) and (6,-1) .. (9,1)
node [pos=0.65, circle, fill, inner sep=1.5pt] {}
        node [pos=0.77, circle, fill, inner sep=1.5pt] {}
        node [pos=0.83, circle, fill, inner sep=1pt] {}
        node [pos=0.865, circle, fill, inner sep=1pt] {}
        node [pos=0.9, circle, fill, inner sep=1pt] {}
        node [pos=0.95, circle, fill, inner sep=1.5pt] {};

\node at (3.95,-4.25) {\small $p$};
\node at (4.75,-3.15) {\small $p_1$};
\node at (5.15,-2.6) {\small $p_2$};
\node at (6.77,-1.2) {\small $p_{d-\alpha}$};
\node[red] at (7.03,-0.85) {\small $x_{s+1}$};
\node[red] at (7.7,-0.4) {\small $x_{s+2}$};
\node[red] at (8.88,0.6) {\small $x_{n}$};
\node[red] at (8.7,1.2) {\small $C_2$};

  \draw [thick,blue,xshift=-2.2,yshift=-18] (0,0) .. controls (2,-1) and (3,-2) .. (4,-4)
        node [pos=0.3, circle, fill, inner sep=1.5pt] {}
        node [pos=0.4, circle, fill, inner sep=1.5pt] {}
        node [pos=0.5, circle, fill, inner sep=1pt] {}
        node [pos=0.55, circle, fill, inner sep=1pt] {}
        node [pos=0.6, circle, fill, inner sep=1pt] {}
        node [pos=0.7, circle, fill, inner sep=1.5pt] {}
        node [pos=0.8, circle, fill, inner sep=1.5pt] {}
        node [pos=0.925, circle, fill=black, inner sep=1.5pt] {};

\node[blue] at (3.2,-3.8) {\small $x_1$};        
\node[blue] at (2.9,-3.3) {\small $x_2$};        
\node[blue] at (1.6,-2.2) {\small $x_{s-1}$};        
\node[blue] at (1.2,-1.6) {\small $x_s$};        
\node[blue] at (0,-1) {\small $C_1$};        

\draw[blue,stealth-stealth] (1.5,0)  -- node[left]{\small $(d-\alpha)B_1$} (3.2,2.2);

\draw [thick,blue,xshift=19,yshift=12] (1.3,-0.7) .. controls (2,-1) and (3,-2) .. (4,-4)
        node [pos=0.3, circle, fill, inner sep=1.5pt] {}
        node [pos=0.4, circle, fill, inner sep=1.5pt] {}
        node [pos=0.5, circle, fill, inner sep=1pt] {}
        node [pos=0.55, circle, fill, inner sep=1pt] {}
        node [pos=0.6, circle, fill, inner sep=1pt] {}
        node [pos=0.7, circle, fill, inner sep=1.5pt] {}
        node [pos=0.8, circle, fill, inner sep=1.5pt] {}
        node [pos=0.915, circle, fill=black, inner sep=1.5pt] {};

\draw [thick,blue,xshift=30,yshift=29] (1.3,-0.7) .. controls (2,-1) and (3,-2) .. (4,-4)
        node [pos=0.3, circle, fill, inner sep=1.5pt] {}
        node [pos=0.4, circle, fill, inner sep=1.5pt] {}
        node [pos=0.5, circle, fill, inner sep=1pt] {}
        node [pos=0.55, circle, fill, inner sep=1pt] {}
        node [pos=0.6, circle, fill, inner sep=1pt] {}
        node [pos=0.7, circle, fill, inner sep=1.5pt] {}
        node [pos=0.8, circle, fill, inner sep=1.5pt] {}
        node [pos=0.925, circle, fill=black, inner sep=1.5pt] {};

\draw [thick,blue,xshift=70,yshift=70] (1.3,-0.7) .. controls (2,-1) and (3,-2) .. (4,-4)
        node [pos=0.3, circle, fill, inner sep=1.5pt] {}
        node [pos=0.4, circle, fill, inner sep=1.5pt] {}
        node [pos=0.5, circle, fill, inner sep=1pt] {}
        node [pos=0.55, circle, fill, inner sep=1pt] {}
        node [pos=0.6, circle, fill, inner sep=1pt] {}
        node [pos=0.7, circle, fill, inner sep=1.5pt] {}
        node [pos=0.8, circle, fill, inner sep=1.5pt] {}
        node [pos=0.915, circle, fill=black, inner sep=1.5pt] {};

\draw [white,xshift=55,yshift=55] (1.3,-0.7) .. controls (2,-1) and (3,-2) .. (4,-4)
node [pos=0.5, circle, fill=blue, inner sep=0.7pt] {}
        node [pos=0.55, circle, fill=blue, inner sep=0.7pt] {}
        node [pos=0.6, circle, fill=blue, inner sep=0.7pt] {}
        node [pos=0.915, circle, fill=black, inner sep=1pt] {};

\draw [white,xshift=50,yshift=50] (1.3,-0.7) .. controls (2,-1) and (3,-2) .. (4,-4)
node [pos=0.5, circle, fill=blue, inner sep=0.7pt] {}
        node [pos=0.55, circle, fill=blue, inner sep=0.7pt] {}
        node [pos=0.6, circle, fill=blue, inner sep=0.7pt] {}
        node [pos=0.915, circle, fill=black, inner sep=1pt] {};
                
 \draw [white,xshift=45,yshift=45] (1.3,-0.7) .. controls (2,-1) and (3,-2) .. (4,-4)
node [pos=0.5, circle, fill=blue, inner sep=0.7pt] {}
        node [pos=0.55, circle, fill=blue, inner sep=0.7pt] {}
        node [pos=0.6, circle, fill=blue, inner sep=0.7pt] {}
           node [pos=0.915, circle, fill=black, inner sep=1pt] {};


  \draw [thick,red] (3,-8.5) .. controls (4.5,-7) and (5.5,-6) .. (8.5,-4)
node [pos=0.65, circle, fill, inner sep=1.5pt] {}
        node [pos=0.77, circle, fill, inner sep=1.5pt] {}
        node [pos=0.83, circle, fill, inner sep=1pt] {}
        node [pos=0.865, circle, fill, inner sep=1pt] {}
        node [pos=0.9, circle, fill, inner sep=1pt] {}
        node [pos=0.95, circle, fill, inner sep=1.5pt] {};

\draw[blue] (0,-5) .. controls (2,-6) and (3,-7) .. (4,-8)
        node [pos=0.3, circle, fill, inner sep=1.5pt] {}
        node [pos=0.4, circle, fill, inner sep=1.5pt] {}
        node [pos=0.5, circle, fill, inner sep=1pt] {}
        node [pos=0.55, circle, fill, inner sep=1pt] {}
        node [pos=0.6, circle, fill, inner sep=1pt] {}
        node [pos=0.7, circle, fill, inner sep=1.5pt] {}
        node [pos=0.8, circle, fill, inner sep=1.5pt] {}
        node [pos=0.915, circle, fill=black, inner sep=1.5pt] {};

\draw[blue,thick,->] (0.5,-1.3) to[out=-110,in=450]node[left]{\small$f_1$} (0.2,-4.7);          

\draw[thick,->] (3.7,-4.5)  to[out=-90,in=450]node[left]{\small$f$} (3.8,-7);          
\draw[red,thick,->] (9,0.4)  to[out=-70,in=450]node[right]{\small$f_2$} (8,-4);          
\node[blue] at (0,-5.25) {\small $B_1$};        
\node[red] at (8.7,-4.15) {\small $B_2$};

\end{tikzpicture}

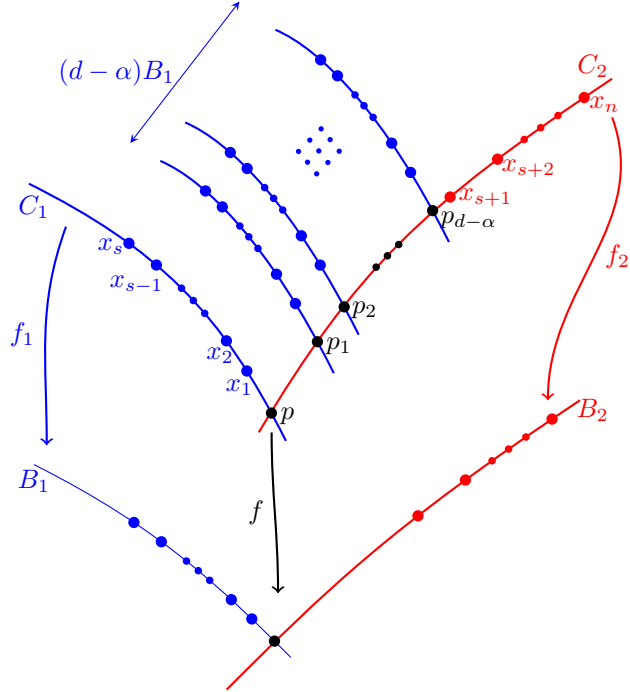
\tiny{\textsc{The curve $X$ mapping to $B=B_1 \cup B_2$}}
\end{center}

 \noindent  Then
  \[ [f: X \to B] \in \overline{\H}_{g \to b,d;e_1,\ldots,e_n},\]
      where $\overline{\H}_{g \to b,d;e_1,\ldots,e_n}$ is the compactification of 
      $\H_{g \to b,d;e_1,\ldots,e_n}$ as explained, e.g., in \cite[p. 3556-7]{ACV} or \cite[Def. 8]{CMR}.

      Now assume that $[f: X \to B]$ is a limit of curves in $\H^{\delta}_{g \to b,d;e_1,\ldots,e_n}$  for some admissible factor $\delta$. Then  $f$ factors into two {\it rational maps} $\overline{\phi}$ of degree  $\delta$  and $\overline{\psi}$ of degree  $\frac{d}{\delta}$  defined in codimension $1$:  
\begin{equation} \label{eq:fac}
  \xymatrix{
    X \ar[rr]^{f} \ar@{-->}[dr]_{\overline{\phi}} &  &  B  \\
   & X' \ar@{-->}[ur]_{\overline{\psi}} & 
  }
\end{equation}
where $X'$ is a reduced curve, and
  \begin{eqnarray*}
    \deg \overline{\phi}|_{X_1} & =\deg \overline{\phi}|_{C_2} & =\delta, \\
    \deg \overline{\psi}|_{\overline{\phi}(X_1)} & =\deg \overline{\psi}|_{\overline{\phi}(C_2)} & =\frac{d}{\delta}.
  \end{eqnarray*}
  Also note that $\overline{\phi}|_{X_1}$ and $\overline{\phi}|_{C_2}$ are everywhere defined, as the domains are smooth.

We first study the restriction of \eqref{eq:fac}  to $C_2$.

  Setting $C_2':=\overline{\phi}(C_2)$, the restriction of  \eqref{eq:fac} to $C_2$ yields a factorization:

\[
\xymatrix{
   C_2 \ar[rr]^{f|_{C_2}=f_2} \ar[dr]_{\overline{\phi}|_{C_2}} &  &  B_2  \\
   & C'_2 \ar@{-->}[ur]_{\overline{\psi}|_{C'_2}} &
   }
      \]
with $\deg \overline{\phi}|_{C_2}=\delta$ and $\deg \overline{\psi}|_{C'_2}=\frac{d}{\delta}$. 
 Now  $[f_2:C_2 \to B_2] \in \H^{\delta}_{g_2\to b,d;\alpha,e_{s+1}, \ldots,e_n}$.  Corollary \ref{cor:claimC} yields that $\alpha \neq d-1$, so that we already have a contradiction if we are in the middle case in  \eqref{eq:defalfa}. Corollary \ref{cor:claimC} also yields that if $\alpha =d-2$, then $\delta=2$. In this case $d$ and all $e_i$ are even (recall that for all $i$ we have either $e_i<\delta$ or $\delta|e_i$). Since  $s=2$ by Remark \ref{rem:minimals}, we have that  $e_1+e_2-2+d$  
 is even, contradicting the assumptions in the lower case  in   \eqref{eq:defalfa}. We conclude that we must be in the upper case in   \eqref{eq:defalfa}.

  We next  study the restriction to $X_1$ of \eqref{eq:fac}. 

    Set $C'_1:=\overline{\phi}(C_1)$, $X'_1:=\overline{\phi}(X_1)$ and consider
    the
    restriction of \eqref{eq:fac} to $X_1$ and $C_1$:

\[
\xymatrix{
C_1 \ar@/_1pc/[dddrr]_{\overline{\phi}|_{C_1}}\ar[rrrr]^{f_1} \ar@{^{(}->} [rd] & & & & B_1 \\
 &  X_1 \ar[rr]^{f|_{X_1}} \ar[dr]_{\overline{\phi}|_{X_1}} &  &  B_1 \ar@{=}[ru] & \\
 &   & X'_1 \ar@{-->}[ur]_{\overline{\psi}|_{X'_1}} & & \\
 &   & C'_1 \ar@{^{(}->}[u] \ar@{-->}@/_1pc/[uuurr]_{\overline{\psi}|_{C'_1}} & &  }
   \]

We separate the  treatment  into three cases:
    \begin{itemize}
    \item[(I)] $\deg \overline{\psi}|_{C'_1}=1$,
    \item[(II)] $\deg \overline{\phi}|_{C_1}=1$,
    \item[(III)] $\deg \overline{\phi}|_{C_1}>1$ and
    $\deg \overline{\psi}|_{C'_1}>1$.  
    \end{itemize}

    {\bf Case I:  $\deg \overline{\psi}|_{C'_1}=1$.} In this case $\overline{\phi}|_{C_1}=f_1$, whence  $\alpha=\deg \overline{\phi}|_{C_1}\leq
    \deg \overline{\phi}|_{X_1}=\delta$, contradicting \eqref{eq:A100}. \medskip
    
 {\bf Case II: $\deg \overline{\phi}|_{C_1}=1$.} In this case $\overline{\phi}|_{C_1}$ is the identity, and $\overline{\psi}|_{C'_1}=f_1$. Since $\deg \overline{\phi}|_{X_1}=\delta$, the only possibility is that $C_1=C'_1\cong B_1$ and
 $\overline{\phi}|_{X_1}$ maps a total of $(d-\alpha+1)$ copies of $B_1$ (including $C_1$) onto $\frac{d-\alpha+1}{\delta}$ copies of $B_1$, which is $X'_1$. Then $\overline{\psi}|_{X'_1}$ maps the $\frac{d-\alpha+1}{\delta}$ copies of $B_1$ to one $B_1$, whence $\frac{d}{\delta}=\deg \overline{\psi}|_{X'_1}=\frac{d-\alpha+1}{\delta}$. But then
 $\overline{\psi}_{X'_1}$ has degree one on each copy, whence $\deg \overline{\psi}|_{C'_1}= \deg f_1=\alpha=1$, 
 contradicting hypothesis \eqref{eq:A5} (recall that  all $e_i  \geq 2$). \medskip

{\bf Case III: $\deg \overline{\phi}|_{C_1}>1$ and
    $\deg \overline{\psi}|_{C'_1}>1$.} 
Set $\alpha_1:=\deg \overline{\phi}|_{C_1}$ and $\alpha_2:=\deg \overline{\phi}|_{C_1}$. Then $\alpha=\alpha_1\alpha_2$. It follows that $\overline{\phi}|_{X_1}$ restricted to the $d-\alpha$ copies of $B_1$ is also of degree $\alpha_1$, so its image is $\frac{d-\alpha}{\alpha_1}$ copies of $B_1$. Hence, $\alpha_1=\delta$.  Moreover, $\delta|\alpha$ and $\delta < \alpha$.

Now $f_1:C_1 \to B_1$ has degree $\alpha$ and has  $s+1\geq 3$  ramification points, one of order $\alpha$. We are therefore in the situation described in Corollary \ref{cor:claimD}, which yields  that $\displaystyle\sum_{i=1}^s(e_i-1) \geq \alpha$. But this contradicts the assumption that $\alpha=\displaystyle\sum_{i=1}^s(e_i-1)+1$ in the upper case in  \eqref{eq:defalfa}.

We have thus reached a contradiction in all cases, proving that $[f: X \to B]$ cannot be a limit of curves in $\H^{\delta}_{g \to b,d;e_1,\ldots,e_n}$  for any admissible factor  $\delta$. 

\subsection{The proof of Theorem \ref{thm:simple01} under assumption
  \eqref{eq:st2} for $b=1$ and of Theorem \ref{thm:simple2}}

 We assume $b \geq 1$. 
Write $b=b_1+b_2$, with 
\[
 \begin{cases}
  b_1,b_2>0, & \mbox{if $b \geq 2$}, \\
  b_1=1,b_2=0, & \mbox{if $b=1$},
\end{cases}
\]

Set
\[
\alpha:= \begin{cases}
  d-1, & \mbox{if $d$ is even}, \\
  d-2, & \mbox{if $d$ is odd},
\end{cases}
\]
so that $\alpha$ is always odd and $2 \leq \alpha<d$.

Define
\begin{eqnarray*}
  g_1 & := &  \frac{\alpha+1}{2}+\alpha(b_1-1), \\
  g_2 & := & d(b_2-1)+\frac{1}{2}\displaystyle\sum_{i=1}^n(e_i-1)+\frac{\alpha+1}{2}.
\end{eqnarray*}
The assumption that $\alpha$ is odd and the fact that  $b_1>0$, yield that $g_1$ is a positive integer. Similarly, $g_2$ is an integer,  and
if $b_2>0$, then also 
$g_2>0$. If $b_2=0$, then  $b=1$,  and assumption \eqref{eq:st2} says that 
$\displaystyle\sum_{i=1}^n(e_i-1)=2g-2 \geq d$, whence 
\begin{eqnarray*}
  g_2  & = & -d+\frac{1}{2}\displaystyle\sum_{i=1}^n(e_i-1)+\frac{\alpha+1}{2} \\
  & = & \frac{1}{2}\left[-2d+\displaystyle\sum_{i=1}^n(e_i-1)+\alpha+1\right] \\
  & \geq & \frac{1}{2}\left[-2d+d+d-1\right] \geq -\frac{1}{2},
\end{eqnarray*}
so that $g_2 \geq 0$ also in this case. 

Proposition \ref{prop:dimhur} now yields the existence of two covers
\[
  [f_1:C_1 \to B_1]   \in \H_{g_1\to b_1,\alpha;\alpha}\]
  and
  \[
    [f_2:C_2 \to B_2] \in  \H_{g_2\to b_2,d;\alpha,e_1,\ldots,e_n}.\] 
We  proceed as in the case $b=0$: Both covers have a  marked  ramification point of order $\alpha$. We glue the two curves $C_1$ and $C_2$ along this point, which we call $p$. We have
\[ f_2^*(f_2(p))=\alpha p+p_1+\cdots+p_{d-\alpha} \;\; \mbox{on} \;\; C_2.\]
Attach one copy of  $B_1$  to $C_2$ at each $p_i$, so that it intersects no other copy of $B_1$ and is disjoint from $C_1$. Define
\begin{eqnarray*}
  X_1 & := & C_1+\mbox{the $(d-\alpha)$ disjoint copies of $B_1$}, \\
  X   & := & X_1+C_2, 
\end{eqnarray*}
so that $X$ looks like in the case $b=0$, this, except that we now do not have any ramification point on $C_1$ except for $p$.
One computes that $p(X)=g_1+g_2+(d-\alpha)b_1=g$. 

Now we glue $B_1$ and $B_2$ along $f_1(p) \in B_1$ and $f_2(p)\in B_2$, and define $B:=B_1 \cup B_2$. Then $p(B)=b_1+b_2$. We can then define a map $f:X \to B$ by gluing $f_1$ and $f_2$ as in the case $b=0$:
  \begin{eqnarray*}
    f|_{C_1} & = & f_1, \\
    f|_{C_2} & = & f_2, \\
    f \; & \mbox{on} & \mbox{each copy of $B_1$ is the identity}.
  \end{eqnarray*}
Then
  \[ [f: X \to B] \in \overline{\H}_{g \to b,d;e_1,\ldots,e_n},\]
      where $\overline{\H}_{g \to b,d;e_1,\ldots,e_n}$ is the compactification of 
      $\H_{g \to b,d;e_1,\ldots,e_n}$ as explained, e.g., in \cite[p. 3556-7]{ACV} or \cite[Def. 8]{CMR}.

      Now assume that $[f: X \to B]$ is a limit of curves in
       $\H^{\delta}_{g \to b,d;e_1,\ldots,e_n}$ for some admissible factor $\delta$.  Then  $f$ factors into two {\it rational maps} $\overline{\phi}$ of degree  $\delta$  and $\overline{\psi}$ of degree  $\frac{d}{\delta}$   defined in codimension $1$:  
\begin{equation} \label{eq:fac2}
  \xymatrix{
    X \ar[rr]^{f} \ar@{-->}[dr]_{\overline{\phi}} &  &  B  \\
   & X' \ar@{-->}[ur]_{\overline{\psi}} & 
  }
\end{equation}
where $X'$ is a reduced curve, and
  \begin{eqnarray*}
    \deg \overline{\phi}|_{X_1} & =\deg \overline{\phi}|_{C_2} & =\delta, \\
    \deg \overline{\psi}|_{\overline{\phi}(X_1)} & =\deg \overline{\psi}|_{\overline{\phi}(C_2)} & =\frac{d}{\delta}.
  \end{eqnarray*}
  Also note that $\overline{\phi}|_{X_1}$ and $\overline{\phi}|_{C_2}$ are everywhere defined, as the domains are smooth.

  The map $f|_{C_2}$ is of degree $d$ with a ramification point of order $\alpha$. Corollary \ref{cor:claimC} says that a factorization $f|_{C_2}= \overline{\psi} \circ \overline{\phi}|_{C_2}$ cannot exist if $\alpha=d-1$, and can only exist when  $\delta=2$, if $\alpha=d-2$. In the latter case $d$ is even (as it is divisible by  $\delta$), but this contradicts the definition of $\alpha$. Thus, a factorization as above cannot exist. This is the desired contradiction, finishing the proofs of Theorems \ref{thm:simple01} and Theorem \ref{thm:simple2}.

  \section{Examples}\label{sec:ex}

In this section we first give some examples of admissible factors and we check how condition $(*)_b$  reads,  explaining the missing cases. Then we give examples highlighting why condition $(*)_b$ cannot be avoided completely. Finally, we give examples admitting components of simple and non-simple type.

  \subsection{Easy examples of admissible factors and condition $(*)_b$}\label{ex-admfac}

\begin{example} \label{ex:d=4}
  Let $d=4$. Then the only possible admissible factor of $(g,b,d;\e)$ is $2$. Setting $a_j:=\#\{i\;|\;e_i=j\}$, for $j \in \{2,3,4\}$, we see that $2$ is an admissible factor if and only if $a_3=0$, $a_4$ is even and, if $b =0$, also $a_4 \geq 2$.  By Theorem \ref{thm:composed}, these are precisely the cases where
$\H_{g \to b,4;\e}$ admits components of non--simple type.  
  Note that the Riemann-Hurwitz formula yields that also $r+a_2$ is even.
 
Condition $(*)_0$ reads $r+a_2 \neq 0$, and $(*)_1$ reads $r+a_2+a_4 \neq 0$. The only cases left open by Theorem \ref{thm:simple01} are therefore the cases where $b=0$ and all $e_i=4$ and where $b=1$ and $n=r=0$.  In all other cases, Theorems \ref{thm:simple2}--\ref{thm:simple01} yield the existence of components of simple type. 
\end{example}

\begin{example} \label{ex:d=6}
  Let $d=6$. Then the only possible admissible factors of $(g,b,d;\e)$ are $2$ and $3$. Set $a_j:=\#\{i\;|\;e_i=j\}$, for $j \in \{2,3,4,5,6\}$. We distinguish between four cases:
  \begin{itemize}
  \item[(i)] The only admissible factor is $2$. This happens if and only if $a_3=a_5=0$, $a_4$ is even and $a_4 \neq 0$, and if $b =0$, also $a_4+2a_6 \geq 4$. Note that the Riemann-Hurwitz formula yields that also $r+a_2+a_6$ is even.
  \item[(ii)] The only admissible factor is $3$. This happens  if and only if $a_4=a_5=0$, $a_3 \neq 0$, $a_6$ is even, and if $b =0$, also $a_6 \geq 2$. Note that the Riemann-Hurwitz formula yields that also $r+a_2$ is even.
  \item[(iii)] Both $2$ and $3$ are admissible factors. This happens  if and only if $a_3=a_4=a_5=0$,  $a_6$ is even, and if $b =0$, also $a_6 \geq 2$. Note that the Riemann-Hurwitz formula yields that also $r+a_2$ is even.
 \item[(iv)] There are no admissible factors.  This happens  if and only if $a_3,a_4 \neq 0$ or $a_5 \neq 0$.   
  \end{itemize}
  By Theorem \ref{thm:composed}, the cases (i)-(iii) are precisely the cases where
$\H_{g \to b,6;\e}$ admits components of non--simple type. By Theorem \ref{thm:simple2}, when $b \geq 2$ there also exist components of simple type. When $b=0,1$,  we now  express  condition $(*)_b$ in the cases (i)-(iii):
  \begin{itemize}
\item[(i)] In this case one can easily check that conditions $(*)_0$ and $(*)_1$ are always satisfied. Therefore, there are no cases left open by Theorem \ref{thm:simple01} which means that $\H_{g \to b,6;\e}$ also has components of simple type in this case, by Theorem \ref{thm:simple01}.  
\item[(ii)] In this case $(*)_0$ reads $r+a_2+a_3 \geq 2$ and  $(*)_1$ reads
  $r+a_2+a_3 +a_6 \geq 2$. Recalling that $a_3 \neq 0$ and $r+a_2$ and $a_6$ are even, and additionally $a_6 \geq 2$ if $b=0$, we see that the only cases left open by Theorem \ref{thm:simple01} are the cases where $b=0$, $r=0$ and $\e=(3,6^{n-1})$ (with $n \geq 3$ and odd), and
 where $b=1$, $r=0$, $n=1$ and $\e=(3)$.  In all other cases, Theorem \ref{thm:simple01} yield the existence of components of simple type. 
\item[(iii)] In this case $(*)_0$ reads $r+a_2 \geq 4$ and  $(*)_1$ reads
 $r+a_2 \geq 4$ or $a_6 \geq 2$.  Recalling that both $r+a_2$ and $a_6$ are even,  and additionally $a_6 \geq 2$ if $b=0$, we see that the only cases left open by Theorem \ref{thm:simple01} are the cases where $b=0$ and
  \begin{eqnarray*} r=0,2, \; \e=6^n \; \mbox{($n \geq 2$ even)}, \;\; \mbox{or} \;\; r=0, \; \e=(2^2,6^{n-2}) \;  \mbox{($n \geq 4$ even)}, \\
\;\; \mbox{or} \;\; r=1, \; \e=(2,6^{n-1}) \; 
\mbox{($n$ odd)},
  \end{eqnarray*}
  and where $b=1$ and
  \[ (r,n)=(0,0),(2,0), \;\; \mbox{or} \;\; (r,n)=(0,2), \; \e=(2,2), \;\; \mbox{or} \;\; (r,n)=(1,1),\;\e=(2). \]
   In all other cases, Theorem \ref{thm:simple01} yields the existence of a component of simple type. 
  \end{itemize}

\end{example}

\begin{example} \label{ex:d=2tot}
  Let us consider the case with two total marked ramification points, that is, $n=2$ and $\e=(d,d)$. Then the remaining unmarked ramification points are $r=2(g-db)$. In this case  the admissible factors of $(g,b,d;\e)$ are precisely the proper nontrivial divisors of $d$.  For each such $\delta$, there exists a component of non--simple type   of $\H_{g\to b, d, (d,d)}$ by Proposition \ref {prop:composed}. For a general point $[\pi: C\to  B]$ in each such component there is a diagram
  \begin{equation*}
  \label{eq:compos}
  \xymatrix{
    C \ar[rr]^{\pi} \ar[dr]_{\phi} &  &   B  \\
   & C' \ar[ur]_{\psi} & 
}\end{equation*}
where $C'$ has genus $g'=\frac{db}{\delta}$, the map $\phi$ has degree $\delta$, with two assigned total ramification points $x,y$ and $r$ further  simple ramification points, and $\psi$ is of degree $d/\delta$ totally ramified at $\phi(x)$ and $\phi(y)$ and with no further ramification. 

The  largest admissible factor $\Delta$ of $(g,b,d;\e)$ is the largest proper  divisor of $d$. 
Condition $(*)_0$ reads $r=2g \geq \Delta$  whereas 
condition $(*)_1$ is always satisfied. 
Thus, the cases left open by Theorem \ref{thm:simple01} are the cases where $b=0$ and $2g<\Delta$.  In all other cases, Theorem \ref{thm:simple01} yields the existence of a component of simple type.

 The situation in the  case $b=0$ is on the other hand known: for $g \geq 1$, the  components of $\H_{g\to 0, d, (d,d)}$ correspond exactly to the divisors of $d$ (including $d$, which gives the unique component of simple type), by  \cite[App.~A]{Agg}, as conjectured in \cite[Rem. 8]{LLV}. For $g=0$, see Example \ref{ex:rat} right below. 
\end{example}

 \subsection{The hypotheses $(*)_b$ cannot be avoided completely} Look in fact at the following examples.  
  
  \begin{example}\label{ex:rat} Consider $\H_{0  \to 0,d; (d,d)}$, i.e., maps $\pi: \mathbb P^1\to \mathbb P^1$ of degree $d$ totally ramified at two points.  According to Proposition \ref {prop:dimhur}, one has $\dim(\H_{0  \to 0,d; (d,d)})=0$. Indeed the two ramification points can be put at $0$ and $\infty$, and the unique (up to isomorphism) map $\pi_d$ in $\H_{0  \to 0,d; (d,d)}$  sends $x$ to $x^d$. If $d$ is prime, then of course $\pi_d$ is simple. Otherwise, if $\delta$ is a proper divisor of $d$, then $\pi$ is composed with $\pi_\delta$ and $\pi_{\frac d \delta}$.  Hence $\H_{0  \to 0,d; (d,d)}$ does not have  any  component of non--simple type   if $d$ is not prime. 
  \end{example} 
  
  \begin{example}\label{ex:ell} Consider $\H_{1  \to 1,d, (0)}$, i.e., unramified maps $\pi: E\to E'$ of degree $d$, with $E,E'$ smooth irreducible curves of genus 1.  According to Proposition \ref {prop:dimhur}, one has $\dim \H_{1  \to 1,d; (0)}=1$. Indeed, given any elliptic curve $E$ (depending on 1 modulus) and given a torsion point $\eta\in E$ of order $d$, there is the morphism $f_{E,\eta}: x \in E\to x+\eta \in E$ of order $d$, which is an action of $\mathbb Z_d$ on $E$. Accordingly we have a quotient morphism $\pi_{E,\eta}: E\to E'$, where $E'$ is the quotient, and $\pi_{E,\eta}$ is not ramified. If  $d$ is prime, then any morphism $\pi_{E,\eta}$ is simple. Otherwise, if $\delta$ is a proper divisor of $d$ then $\pi_{E,\eta}$ is composed with $\pi_{E,\delta \eta}$.  Hence $\H_{1  \to 1,d, (0)}$ does not have  any   component of non--simple type  if $d$ is not prime. 
  \end{example} 
  
  \subsection{ Examples  with  components of both simple and non--simple type}  The following easy example shows instead a situation in which we have components of $\H_{g \to b,d;\e}$ that are  of simple and of non--simple type.
  
  \begin{example} \label{ex:bot} Let us consider $\H_{g \to 0,4; (4,4)}$, with $g\geq 1$.  According to Proposition \ref {prop:dimhur}, one has $\dim \H_{g  \to 0,4; (4,4)}=2g-1$.  One  component of non--simple type  of $\H_{g  \to 0,4; (4,4)}$ is  obtained  in the following way. Take any hyperelliptic curve $C$ (depending on $2g-1$ moduli). This has the hyperelliptic degree 2 morphism $\phi: C\to \mathbb P^1$ with $2g+2\geq 4$ branch points. Compose it with a degree 2 morphism $\psi: \mathbb P^1\to \mathbb P^1$ ramified at two of the $2g+2$ branch points of $\phi$. The composition belongs to $\H_{g  \to 0,4; (4,4)}$  and depends on $2g-1$ parameters, so when $C$ varies in the hyperelliptic locus, we have a whole component of $\H_{g  \to 0,4; (4,4)}$. Actually, according to Proposition \ref {prop:composed}, this is the only  component of non--simple type  of $\H_{g  \to 0,4; (4,4)}$. 
 This also follows from \cite[App.~A]{Agg}.   
  
    On the other hand,  since the largest admissible factor of 
$(g,0,4,(4,4))$ is $2$ and the condition $2g \geq 2$ is satisfied by the assumption $g \geq 1$, then  according to Theorem \ref {thm:simple01},
there is a
component of simple type  of  the space  $\H_{g  \to 0,4; (4,4)}$  (cf. Example \ref{ex:d=2tot}).   Actually, by \cite[App.~A]{Agg}, there is a unique such component. 

 By the above,  for $g \geq 3$,  the existence of a component of simple type  would follow from    the fact  that there is a point $[\pi: C\to \mathbb P^1]$ in $\H_{g  \to 0,4; (4,4)}$ with $C$ non--hyperelliptic. This is not immediate in general. However it is easy  to verify this directly  for low values of $g$,  and at the same time prove the existence of a {\it unique  and unirational}  component of simple type.  For example for $g=3$, using Bertini's theorem, one can find a smooth plane quartic with two tangents with contact of order 4 at two distinct points,  and one easily verifies that the family of such curves is irreducible  and rational.  
  
  For $4\leq g\leq 5$ a similar argument works, and we will briefly explain it.
  
   For $g=4$ to describe the points $[\pi: C\to \mathbb P^1]$ in $\H_{g  \to 0,4; (4,4)}$ with $C$ non--hyperelliptic, one can proceed as follows.  Given a non--hyperelliptic curve $C$ of genus 4, and a $g^1_4$ on $C$, choose a base point free $g^1_3$ on $C$. Using the $g^1_4$ and the $g^1_3$, we can map  $C$  birationally  to $\mathbb P^1\times \mathbb P^1$, with image $\Gamma$, of type $(3,4)$ on $\mathbb P^1\times \mathbb P^1$. We have $\dim (|\Gamma|)=19$ and $p_a(\Gamma)=6$.   Note that the rulings of type $(1,0)$ cut out on $\Gamma$ a linear series that pulls back to the $g^1_4$ on $C$ and by definition this $g^1_4$ has two points of total ramification. 
   
 Choose  two general  points $p,q$ on $\mathbb P^1\times \mathbb P^1$. The linear system $\mathcal L_{p,q}$ of curves of type $(3,4)$ with double points at $p,q$ has dimension at least  $13$. Next fix two more general points $x,x'$ on $\mathbb P^1\times \mathbb P^1$, and denote by $F,F'$ the two rulings of type $(1,0)$ containing $x$ and $x'$ respectively. Then impose that the curves in $\mathcal L_{p,q}$ have intersection multiplicity at least $4$ with $F$ and $F'$ at $x$ and $x'$ respectively. In this way one gets a linear system $\mathcal L_{p,q,x,x'}$ of dimension at least $5$. Next we claim that: 
 
 \begin{eqnarray}\label{eq:haaland}
 & \dim (\mathcal L_{p,q,x,x'})=5,\,\, \text{and its general element is irreducible} & \\
  \nonumber & \text{with nodes at $p,q$
  and  smooth at $x,x'$.} &
 \end{eqnarray}
 Indeed, 
 by imposing to the elements of $\mathcal L_{p,q,x,x'}$ to contain two general points of $F,F'$, we split off these two fibres from $\mathcal L_{p,q,x,x'}$, and the residual system $\mathcal L'_{p,q,x,x'}$ has dimension at least 3 and consists of curves of type $(1,4)$ with double points at $p,q$. Then the two fibres of type $(0,1)$ through $p,q$ split off the system $\mathcal L'_{p,q,x,x'}$ and the residual system consists of curves of type $(1,2)$ containing $p,q$. This system has exactly dimension 3, and the irreducibility of its general member easily follows by Bertini's theorem. Then
 \eqref {eq:haaland} immediately follows.

   This proves that the variety $\mathcal V$ of curves of type $(3,4)$ on $\mathbb P^1\times \mathbb P^1$, with genus 4, and  such that the rulings of type $(1,0)$ cut out  a linear series that has two total ramification points, is nonemptyand dominates the simple locus of $\H_{4  \to 0,4; (4,4)}$.

   By results of Caporaso and Harris \cite{CH},  worked out in a more general setting in  \cite{De},  the variety $\mathcal V$ is equidimensional of dimension 13 and the general element in any irreducible component of $\mathcal V$ has two nodes $p$ and $q$ and is smooth and totally tangent to two rulings of type $(1,0)$ at two points $x,x'$. 
   
   Now we claim that $\mathcal V$ is irreducible. Indeed, given a general element $\Gamma$ in an irreducible component of $\mathcal V$, we can associate to $\Gamma$ the four points $p,q$ (the two nodes) and $x,x'$ (the two tangency points). In this way we have a map 
   $$\mathcal V\dasharrow {\rm Sym}^2(\mathbb P^1\times \mathbb P^1)^2$$
 and, with the same argument we made above to prove \eqref{eq:haaland}, one sees that  the fibre over any quadruple of points  $p,q,x,x'$ in the image of the map is the linear system $\mathcal L_{p,q,x,x'}$ of dimension 5. Since $\mathcal V$ is equidimensional of dimension 13, this proves that any component of $\mathcal V$ dominates  ${\rm Sym}^2(\mathbb P^1\times \mathbb P^1)$, and since the general fibre is irreducible, this proves the claim that $\mathcal V$ is irreducible. 
 It is clearly rational.

 As a consequence, there  is  only one component of $\H_{4  \to 0,4; (4,4)}$ of simple type  and it is unirational.

 The situation for $g=5$ is similar. Take a  non--hyperelliptic curve $C$ of genus 4, and a $g^1_4$ on $C$.  Choose another base point free $g^1_4$ on $C$ (that varies with one parameter). Using the two $g^1_4$s, we can map  $C$ to $\mathbb P^1\times \mathbb P^1$ with image $\Gamma$, a birational curve to $C$ of type $(4,4)$ on $\mathbb P^1\times \mathbb P^1$. We have $\dim (|\Gamma|)=24$ and $p_a(\Gamma)=9$, so that $\Gamma$  has singularities that drop the genus by 4.  
 
 Choose now four general  points $p_1,\ldots, p_4$ on $\mathbb P^1\times \mathbb P^1$. The linear system $\mathcal L_{p_1,\ldots, p_4}$ of curves of type $(4,4)$ with double points at $p_1,\ldots, p_4$ has dimension $12$. Again fix two more general points $x,x'$ on $\mathbb P^1\times \mathbb P^1$, and denote by $F,F'$ the two rulings containing $x$ and $x'$ respectively, such that $F\cdot \Gamma=F'\cdot \Gamma=4$. Then impose that the curves in $\mathcal L_{p_1,\ldots, p_4}$ have intersection multiplicity $4$ with $F$ and $F'$ at $x$ and $x'$ respectively.  As in the case $g=4$, one proves that the linear system $\mathcal L_{p_1,\ldots, p_4,x,x'}$ has dimension $4$ and its general element is irreducible, with  nodes at $p_1,\ldots, p_4$ and  smooth at $x,x'$. 
 
  This proves again that the  variety $\mathcal V$ of curves of type $(4,4)$ on $\mathbb P^1\times \mathbb P^1$, with genus 5, and  such that the rulings of type $(1,0)$ cut out  a linear series that has two total ramification points, is nonempty and dominates the simple locus of $\H_{5  \to 0,4; (4,4)}$. 
  
  Again by \cite{CH, De} the variety $\mathcal V$ is equidimensional of dimension 16 and the general element in any irreducible component of $\mathcal V$ has four nodes $p_1,\ldots, p_4$ and  is  totally tangent to two rulings of type $(1,0)$ at two smooth points $x,x'$. 
  
  Then the irreducibility of  $\mathcal V$ is proved similarly as in the $g=4$ case, and this implies that there  is  only one component of $\H_{5  \to 0,4; (4,4)}$ of simple type  and that it is unirational. 
 
  The situation becomes more complicated  for $g>5$. 
\end{example}  
  
 The next example is related to  Example \ref{ex:d=2tot} and shows that Theorem \ref{thm:simple01} can be improved in some cases.
  
  \begin{example}\label{ex:sim}
Let us consider $\mathcal H_{1\to 0, d, (d,d)}$,  which  has dimension 1. Any proper factor $\delta$ of $d$ is an admissible factor of $(1,0,d;(d,d))$, and accordingly, for each such $\delta$, we have a
 component of non--simple type   of $\mathcal H_{1\to 0, d, (d,d)}$ (see Proposition \ref {prop:composed}). Any such component can be described in the following way. Its general point $[\pi: C\to \mathbb P^1]$ is such that there is a diagram
  \begin{equation*}
  \label{eq:compos}
  \xymatrix{
    C \ar[rr]^{\pi} \ar[dr]_{\phi} &  &  \mathbb P^1  \\
   & \mathbb P^1 \ar[ur]_{\psi} & 
}\end{equation*}
where $\phi$ has degree $\delta$, with two assigned total ramification points $x,y$ and two further  simple ramification points, and $\psi$ is of degree $d/\delta$ totally ramified at $\phi(x)$ and $\phi(y)$.

 Theorem \ref{thm:simple01} does not guarantee the existence of a 
 component of simple type of $\mathcal H_{1\to 0, d, (d,d)}$  when $d>4$,  since the condition $2g \geq \Delta$ is not satisfied (cf. Example \ref{ex:d=2tot}). However,  the existence of such a component follows from  \cite[App.~A]{Agg}, but one can also prove this  directly.   Indeed, given a general elliptic curve $C$, choose two distinct points $x,y\in C$ such that $x-y$ is a point of order $d$ but not of order smaller than $d$. Then $dx$ is linearly equivalent to $dy$ and therefore there is a $g^1_d$ on $C$ totally ramified at $x$ and $y$. By the above considerations, the corresponding  point $[\pi: C\to \mathbb P^1]$ is simple.   \end{example}

  \end{document}